\documentclass[11pt,a4paper,oneside]{amsart}
\usepackage{fullpage}
\usepackage{amssymb,amsfonts}
\usepackage{epsfig}
\usepackage{graphicx}

\usepackage{graphics}
\newtheorem{theorem}{Theorem}[section]
\newtheorem{lemma}[theorem]{Lemma}

\theoremstyle{definition}

\newtheorem{remark}[theorem]{Remark}

\numberwithin{equation}{section}

\newcommand{\be}{\begin{equation}}
\newcommand{\ee}{\end{equation}}

\DeclareMathOperator{\spt}{spt}

\makeindex

\def\XXint#1#2#3{{\setbox0=\hbox{$#1{#2#3}{\int}$}
 \vcenter{\hbox{$#2#3$}}\kern-.5\wd0}}

\usepackage[colorinlistoftodos]{todonotes}
\usepackage[colorlinks=true, allcolors=blue]{hyperref}

\title{A necessary condition on domains for optimal Orlicz-Sobolev embeddings}
\author{Nijjwal Karak}
\address{Department of Mathematical Analysis, Charles University, Sokolovsk\'a 83, 18600 Prague 8, Czech Republic}
\email{nijjwal@gmail.com}
\thanks{This work was supported by OP RDE project no. CZ.02.2.69/0.0/0.0/16\_027/0008495, International Mobility of Researcher at Charles University.}
\begin{document}
\begin{abstract}
We provide a necessary condition on the regularity of domains for the optimal embeddings of first order (and higher order) Orlicz-Sobolev spaces into Orlicz spaces in the sense of \cite{Cia96} (and \cite{Cia06}).
\end{abstract}
\maketitle
\section{Introduction}
This article concerns with a necessary condition for optimal Orlicz-Sobolev embeddings, when the functions do not necessarily vanish on the boundary, in terms of the regularity of the domain. The first work in this direction was the result of Haj\l asz-Koskela-Tuominen \cite{HKT08b} for optimal Sobolev embeddings, in all possible cases. Later, there have been more results for other Sobolev-type embeddings, namely for fractional Sobolev spaces, variable exponent Sobolev spaces, Besov and Triebel-Lizorkin spaces in $\mathbb{R}^n$ and in metric measure spaces, see \cite{Gor17, GKP, HIT16, Kar2, Kar1, Zho15}. This seems to be the first attempt for Orlicz-Sobolev spaces and the regularity of the domain which we concern is so-called the measure density condition. A subset $\Omega$ of $\mathbb{R}^n,$ $n\geq 2,$ is said to satisfy the measure density condition if there exists a positive constant $c$ such that, for all $x\in\Omega$ and all $0<r\leq 1,$
\begin{equation}
\vert B(x,r)\cap\Omega\vert\geq cr^n.
\end{equation}
We use the notation $\vert S\vert$ for the Lebesgue measure of a set $S.$ Note that sets satisfying measure density condition are sometimes called in the literature regular, Ahlfors $n$-regular or $n$-sets, \cite{JW84, Shv07}. Also note that sets satisfying such a condition have zero boundary measure, \cite[Lemma 2.1]{Shv07}. Some examples of sets satisfying the measure density condition are Cantor-like sets such as
Sierpi\'nski carpets (or gaskets) of positive measure.\\

\noindent In \cite{Cia96, Cia06}, the author has established, for a Young function $A,$ optimal embedding for $W^{m,A}(\Omega)$ with $\Omega\in\mathcal{G}(1/n')$ for first order Orlicz-Sobolev spaces and with $\Omega$ having Lipschitz boundary for higher-order Orlicz-Sobolev spaces. Here and in what follows $n\geq 2,$ $n'=\frac{n}{n-1}$ and
\begin{multline*}
\mathcal{G}(1/n')=\{G\in\mathbb{R}^n:G\,\,\text{is open and}\,\, N>0, Q>0\,\,\text{exist such that}\\
\vert E\vert^{\frac{1}{n'}}\leq QP(E;G) \,\,\text{for all}\,\, E \,\,\text{such that}\,\, \vert E\vert\leq N\},
\end{multline*}
$P(E;G)$ denotes the perimeter of $E$ relative to $G,$ \cite{Maz85}. A bounded domain $\Omega$ is called a Lipschitz domain if each point $x\in\partial\Omega$ has a neighborhood $U$ such that $U\cap\Omega$ is represented by the inequality $x_n<f(x_1,\ldots,x_{n-1})$ in some Cartesian coordinate system with function $f$ satisfying a Lipschtiz condition, \cite{Maz85}. It is easy to see that both these domains satisfy measure density condition; see \cite{AF03, Maz85} for more details about regularity of domains and their relations.\\ 

\noindent Let $A$ be any Young function such that $\int_0 \tilde{A}(t)/t^{1+n'}\,dt<\infty$ and let $A_n$ be the Young function defined by
\begin{equation}\label{A_n}
A_n(s)=\int_0^s r^{n'-1}(\Phi_n^{-1}(r^{n'}))^{n'}\,dr,
\end{equation}
where $\Phi_n^{-1}$ is the generalized left-continuous inverse of
\begin{equation}\label{Phi_n}
\Phi_n(s)=\int_0^s \frac{\tilde{A}(t)}{t^{1+n'}}\,dt.
\end{equation}
It is proved in \cite{Cia96} that if $\Omega\in\mathcal{G}(1/n'),$ then the continuous embedding
\begin{equation}\label{ctsembedding}
W^{1,A}(\Omega)\hookrightarrow L^{\bar{A}_n}(\Omega)
\end{equation}
holds, where $\bar{A}_n$ is the Young function defined by
\begin{equation*}
\bar{A}_n(s)=
\begin{cases}
A_n(s)& \text{if $s\geq s_2$}\\
A(s)& \text{if $0\leq s\leq s_1,$}
\end{cases}
\end{equation*}
for some $0<s_1<s_2.$ Moreover, this embedding is optimal in the sense that $L^{\bar{A}_n}(\Omega)$ is the smallest Orlicz space that renders \eqref{ctsembedding} true. In this article, we have proved that for any open set $\Omega,$ if the above embedding holds, with some additional restriction on $A,$ then $\Omega$ satisfies the measure density condition, see Theorem \ref{main}.\\

\noindent Let $1\leq m<n$ and $A$ be any Young function satisfying 
\begin{equation}\label{finite}
\int_0\left(\frac{s}{A(s)}\right)^{\frac{m}{n-m}}\,ds<\infty.
\end{equation} 
Note that the conditions \eqref{finite} and $\int_0 \tilde{A}(t)/t^{1+n'}\,dt<\infty$ are equivalent, \cite[Lemma 2]{Cia00}.  Let us define, for $s\geq 0,$
\begin{equation}\label{A_nm}
A_{\frac{n}{m}}(s)=A\circ H_{\frac{n}{m}}^{-1}(s),
\end{equation}
where $H_{\frac{n}{m}}^{-1}$ denotes the generalized left-continuous inverse of
\begin{equation*}
H_{\frac{n}{m}}(r)=\left(\int_0^r\left(\frac{t}{A(t)}\right)^{\frac{m}{n-m}}\,dt\right)^{\frac{n-m}{n}}, \quad\text{for}\,\, r\geq 0.
\end{equation*}
Cianchi \cite{Cia06} has proved that if $\Omega$ is a Lipschitz domain then the continuous embedding
\begin{equation}\label{ctsembeddinghigher}
W^{m,A}(\Omega)\hookrightarrow L^{A_{\frac{n}{m}}}(\Omega)
\end{equation}
holds and the embedding is optimal. Analogous to Theorem \ref{main}, we have proved a similar result here in this article, see Theorem \ref{main-higher}.
\section{Preliminaries}
\subsection{Young functions}
A function $A:[0,\infty)\rightarrow [0,\infty]$ is called a Young function if it has the form
\begin{align*}
A(s)=\int_0^s a(r)\, dr,
\end{align*}
where $a:[0,\infty)\rightarrow [0,\infty]$ with $a(0)=0,$ is an increasing, left-continuous function which is neither identically zero nor identically infinite on $(0,\infty).$ Every Young function is non-negative, increasing, left-continuous and convex on $[0,\infty).$ Moreover, $A(s)/s$ is increasing on $(0,\infty)$ and we have
\begin{align*}
\frac{A(s)}{s}\leq a(s)\leq \frac{A(2s)}{s},\quad\text{for}\,\, s>0. 
\end{align*}
Observe that the function $\bar{A}$ defined by $\bar{A}(s)=\int_0^sA(r)/r\,dr$ is also a Young function and satisfies
\begin{equation}\label{compare}
\bar{A}(s)\leq A(s)\leq \bar{A}(2s), \quad\text{for}\,\, s\geq 0.
\end{equation}
The Young conjugate $\tilde{A}$ of $A$ is given by
\begin{equation}
\tilde{A}(s)=\sup\{sr-A(r):r>0\}.
\end{equation} 
Note that $\tilde{\tilde{A}}=A.$ Also for any Young function $A,$ and the right-continuous inverses $A^{-1},$ $\tilde{A}^{-1}$ of $A$ and $\tilde{A}$ respectively, one has, for $0\leq r\leq \infty,$
\begin{equation}\label{relation}
r\leq A^{-1}(r)\tilde{A}^{-1}(r)\leq 2r.
\end{equation}
Given two Young functions $A$ and $B,$ the function $B$ is said to dominate the function $A$ globally [resp. near infinity] if a positive constant $c$ exists such that $A(s)\leq B(cS)$ holds for $s\geq 0$ [resp. for s greater than some positive number]. The functions $A$ and $B$ are called equivalent globally [near infinity] if each dominates the other globally [near infinity]. For more details about Young functions and their properties, see, for example, \cite{PKAO13, RR91}.
\subsection{Orlicz spaces}
Let $\Omega$ be a measurable subset of $\mathbb{R}^n.$ For a Young function $A,$ the Orlicz space $L^A(\Omega)$ is the collection of all measurable functions $f$ on $\Omega$ such that
\begin{equation}
\int_{\Omega}A\left(\frac{\vert f(x)\vert}{\lambda}\right)\,dx<\infty
\end{equation} 
for some $\lambda>0.$ The Orlicz space $L^A(\Omega)$ is a Banach space endowed with the Luxemburg norm
\begin{equation*}
\Vert f\Vert_{L^A(\Omega)}=\inf\left\{\lambda>0:\int_{\Omega}A\left(\frac{\vert f(x)\vert}{\lambda}\right)\,dx\leq 1\right\}
\end{equation*}
for a measurable function $f$ on $\Omega.$ Notice that for a measurable set $E$ in $\Omega$ with positive measure, we have
\begin{equation}\label{chi_norm}
\Vert \chi_E\Vert_{L^A(\Omega)}=\frac{1}{A^{-1}(\frac{1}{\vert E\cap \Omega\vert})},
\end{equation}
where $\chi_E$ denotes the characteristic function of $E.$ We refer to \cite{AF03, PKAO13, RR91} for more details about Orlicz spaces.
\subsection{Orlicz-Sobolev spaces}
Let $\Omega\subset\mathbb{R}^n$ be open and $m$ be a positive integer. The $m$-th order Orlicz-Sobolev space $W^{m,A}(\Omega)$ is defined as
\begin{multline*}
W^{m,A}(\Omega)=\{u\in L^A(\Omega): u \,\,\text{is}\,\,m\text{-times weakly differentiable in}\,\Omega\,\,\text{and}\\
\nabla^{\alpha}u\in L^A(\Omega)\,\,\text{for every}\,\, \alpha \,\,\text{such that}\,\, \vert\alpha\vert\leq m\}.
\end{multline*}
Here $\alpha$ is any multi-index having the form $\alpha=(\alpha_1,\ldots,\alpha_n)$ for non-negative integers $\alpha_1,\ldots,\alpha_n,$ $\vert\alpha\vert=\alpha_1+\cdots+\alpha_n$ and $\nabla^{\alpha}u=\frac{\partial^{\vert\alpha\vert} u}{\partial x_1^{\alpha_1}\cdots\partial x_n^{\alpha_n}}.$ The space $W^{m,A}(G)$ is a Banach space equipped with the norm $\Vert u\Vert_{W^{m,A}(\Omega)}=\sum_{\vert\alpha\vert\leq m}\Vert \nabla^{\alpha}u\Vert_{L^A(\Omega)}.$ More details about this subsection can be found in \cite{AF03, Tuo04}.
\subsection{Boyd index} A local upper Boyd index $I_A$ of a Young function $A$ is defined as
\begin{equation*}
I_A=\inf_{0<t<1}\frac{\log t}{\log h_A(t)},
\end{equation*}
where the function $h_A:(0,\infty)\rightarrow [0,\infty]$ is given by
\begin{equation*}
h_A(t)=\limsup_{s\rightarrow\infty}\frac{A^{-1}(st)}{A^{-1}(s)},\quad\text{for}\,\, t>0.
\end{equation*}
We refer to \cite{Boy71} for the definitions of other Boyd indices and more details about them. We will need the following lemma regarding equivalency of pointwise growth, integral growth and the local upper Boyd index of a Young function, see \cite{Cia99} or \cite{CM} for the proof and for other equivalent conditions.
\begin{lemma}\label{musil-cianchi}
Let $A$ be a Young function and let $0<\alpha<1.$ Then the following conditions are equivalent.\\
$(i)$ The local upper Boyd index satisfies $I_A<1/\alpha.$\\
$(ii)$ There exists a constant $k>1$ such that
$$\int_1^t\frac{\tilde{A}(s)}{s^{1/(1-\alpha)+1}}\,ds\leq\frac{\tilde{A}(kt)}{t^{1/(1-\alpha)}}\quad\text{near infinty}.$$
$(iii)$ There exist constants $\sigma>1$ and $c\in (0,1)$ such that
$$A(\sigma t)\leq c\sigma^{1/\alpha}A(t)\quad\text{near infinty}.$$
\end{lemma}
\section{Main results}
\noindent Here is our main Theorem:
\begin{theorem}\label{main}
Let $A$ be a Young function such that $\int_0 \tilde{A}(t)/t^{1+n'}\,dt<\infty$ and $I_A<n.$ 
Let $A_n$ be the Young function defined by \eqref{A_n} and $\Omega$ be any open subset of $\mathbb{R}^n.$ Assume that the continuous embedding $W^{1,A}(\Omega)\hookrightarrow L^{\bar{A}_n}(\Omega)$ holds, where $\bar{A}_n$ is equivalent to $A_n$ near infinity. Then $\Omega$ satisfies the measure density condition.
\end{theorem}
\begin{remark}
The condition $\int_0 \tilde{A}(t)/t^{1+n'}\,dt<\infty$ does not seem to be so restrictive; $A$ can be modified, if necessary, near zero in such a way that $\int_0 \tilde{A}(t)/t^{1+n'}\,dt<\infty$ and it leaves the space $W^{1,A}(\Omega)$ unchanged whenever $\vert\Omega\vert<\infty.$ However, the condition $I_A<n$ seems to be an extra condition here and we do not know how to get rid of it. Also note that, the function $t^p\log^{\lambda}(e+t)$ satisfies the condition $I_A<n$ whenever $1\leq p\leq n.$ 
\end{remark}
\begin{proof}
For a fixed $x\in\Omega$ and for any $0<R\leq 1,$ denote $B_R=B(x,R)\cap\Omega.$ We take $\tilde{R}<R,$ the smallest real number such that
\begin{equation}\label{half_A_R}
\vert B_{\tilde{R}}\vert=\frac{1}{2}\vert B_R\vert.
\end{equation}

For a fixed $x\in\Omega,$ let $u(y):=\eta(y-x)$ be a function of $y\in\Omega$ where $\eta$ is a cut-off function satisfying:\\
1. $\eta: \mathbb{R}^n\rightarrow [0,1],$\\
2. $\spt \eta \subset B_R,$\\
3. $\eta\vert_{B_{\tilde{R}}}=1,$ and\\
4. $\vert\nabla \eta\vert\leq \tilde{c}/(R-\tilde{R})$ for some constant $\tilde{c}.$\\
Now observe that,
\begin{equation*}
\Vert u\Vert_{L^{\bar{A}_n}(\Omega)}\geq \Vert \chi_{B_{\tilde{R}}}\Vert_{L^{\bar{A}_n}(\Omega)}
\end{equation*}
and
\begin{eqnarray*}
\Vert u\Vert_{W^{1,A}(\Omega)} &=& \Vert u\Vert_{L^A(\Omega)}+\Vert \nabla u\Vert_{L^A(\Omega)}\\
&\leq & \Vert \chi_{B_R}\Vert_{L^A(\Omega)}+\frac{\tilde{c}}{R-\tilde{R}}\Vert \chi_{B_R\setminus B_{\tilde{R}}}\Vert_{L^A(\Omega)}\\
&\leq &\frac{2\max\{1,\tilde{c}\}}{R-\tilde{R}}\Vert\chi_{B_R}\Vert_{L^A(\Omega)}.
\end{eqnarray*}
Therefore, from the embedding
\begin{equation}\label{embedding}
\Vert u\Vert_{L^{\bar{A}_n(\Omega)}}\leq c_e\Vert u\Vert_{W^{1,A}(\Omega)},
\end{equation}
we obtain
\begin{equation*}
R-\tilde{R}\leq 2c_e\max\{1,\tilde{c}\}\frac{\Vert\chi_{B_R}\Vert_{L^A(\Omega)}}{\Vert \chi_{B_{\tilde{R}}}\Vert_{L^{\bar{A}_n}(\Omega)}}.
\end{equation*}
Then we use \eqref{chi_norm} and \eqref{half_A_R} to get
\begin{equation}\label{R-}
R-\tilde{R}\leq 2c_e\max\{1,\tilde{c}\}\frac{\bar{A}_n^{-1}\left(\frac{2}{\vert B_{R}\vert}\right)}{A^{-1}\left(\frac{1}{\vert B_R\vert}\right)},
\end{equation}
where $A^{-1}$ and $\bar{A_n}^{-1}$ are the right-continuous inverses of $A$ and $\bar{A}_n$ respectively.\\
Now, we need the following Lemma. The construction of the functions in the proof of the Lemma is due to \cite{Cia96}.
\begin{lemma}\label{mainlemma}
Let $A$ and $\bar{A}_n$ be the same as Theorem \ref{main}. 
Then constants $r_0$ and $c_0$ exist such that
\begin{equation}
\frac{\bar{A}_n^{-1}(2r)}{A^{-1}(r)}\leq c_0\left(\frac{1}{r}\right)^{1/n},
\end{equation}
for all $r\geq r_0.$
\end{lemma}
\begin{proof}
Let $A_n$ and $\Phi_n$ be the Young functions defined by \eqref{A_n} and \eqref{Phi_n} respectively. Set
\begin{equation}
C_n(s)=s^{n'}(\Phi_n^{-1}(s^{n'}))^{n'}.
\end{equation}
Then, we have by \eqref{compare}, for $s>0,$
\begin{equation}
C_n(\frac{s}{2})\leq A_n(s)\leq C_n(s)
\end{equation}
and hence, for $r>0,$
\begin{equation}\label{ACinverse}
\frac{1}{2}A_n^{-1}(r)\leq C_n^{-1}(r)\leq A_n^{-1}(r).
\end{equation}
Moreover, on setting, for $s>0,$
\begin{equation}
D_n(s)=s^{n'}\Phi_n(s),
\end{equation}
we get, for $r>0,$
\begin{equation}\label{Cinverse}
C_n^{-1}(r)=\frac{r^{1/n'}}{D_n^{-1}(r)}.
\end{equation}
Here $C_n^{-1}$ and $D_n^{-1}$ are the right-continuous inverses of $C_n$ and $D_n$ respectively. Since $I_A<n,$ we have from Lemma \ref{musil-cianchi} that constants $s_0$ and $k_0$ exist such that
\begin{equation}
\int_1^s \frac{\tilde{A}(t)}{t^{1+n'}}\,dt\leq \frac{\tilde{A}(k_0s)}{s^{n'}},\quad\text{for}\,\,s\geq s_0.
\end{equation}
Let us choose $s_0>2.$ Then $\int_1^s \frac{\tilde{A}(t)}{t^{1+n'}}\,dt$ and $\int_0^s \frac{\tilde{A}(t)}{t^{1+n'}}\,dt$ are comparable. Consequently, a constant $c_1>0$ exists such that
\begin{equation}
D_n(s)\leq c_1\tilde{A}(k_0s), \quad\text{for}\,\,s\geq s_0.
\end{equation}
Hence there exist constants $r_0$ and $c_2>1$ such that, for $r\geq r_0,$
\begin{equation}
\tilde{A}^{-1}\left(\frac{r}{c_2}\right)\leq k_0 D_n^{-1}(r),
\end{equation}
which yields, after using \eqref{relation} and the increasing property of both $A^{-1}$ and $\tilde{A}^{-1},$ for $r\geq r_0/2$,
\begin{equation}\label{Ainverse}
\frac{r}{c_2A^{-1}(r)}\leq \frac{r}{c_2A^{-1}(r/c_2)}\leq \tilde{A}^{-1}\left(\frac{r}{c_2}\right)\leq \tilde{A}^{-1}\left(\frac{2r}{c_2}\right) \leq k_0 D_n^{-1}(2r).
\end{equation}
Therefore, following \eqref{ACinverse}, \eqref{Cinverse} and \eqref{Ainverse}, we obtain, for $r\geq r_0/2$,
\begin{equation}
\frac{A_n^{-1}(2r)}{A^{-1}(r)}\leq 2^{1+1/n'}k_0c_2\left(\frac{1}{r}\right)^{1/n}.
\end{equation}
Now, since $A_n$ and $\bar{A}_n$ are equivalent near infinity, we get the desired estimate by choosing $r_0$ big enough.
\end{proof}
To finish the proof of Theorem \ref{main}, let us fix $r_0$ from the above Lemma. It is enough to consider the case when $\vert B_R\vert\leq \frac{1}{r_0},$ otherwise $\vert B_R\vert>\frac{1}{r_0}\geq \frac{R^n}{r_0}$ and there is nothing to prove. Then the above lemma yields
\begin{equation}
\frac{\bar{A}_n^{-1}\left(\frac{2}{\vert B_{R}\vert}\right)}{A^{-1}\left(\frac{1}{\vert B_R\vert}\right)}\leq c_0\vert B_R\vert^{1/n}
\end{equation}
and hence \eqref{R-} becomes
\begin{equation}\label{R-final}
R-\tilde{R}\leq c_3\vert B_R\vert^{1/n},
\end{equation} 
where $c_3=2c_ec_0\max\{1,\tilde{c}\}.$\\

To conclude the proof, construct a sequence $R_i$ by setting $R_0:=R$ and $R_{i+1}:=\tilde{R}_i$ inductively for $i\geq 0.$ It follows that
\begin{equation}
\vert B_{R_{i+1}}\vert=\frac{1}{2}\vert B_{R_i}\vert
\end{equation}
with $\lim_{i\rightarrow\infty}R_i=0.$\\

By applying the above method, we observe that
\begin{equation}
R_i-R_{i+1}\leq c_3\vert B_{R_i}\vert^{1/n}=c_32^{-i/n}\vert B_R\vert^{1/n}
\end{equation}
and deduce
\begin{equation}
R=\sum_{i=0}^{\infty}(R_i-R_{i+1})\leq \frac{c_3}{1-2^{-1/n}}\vert B_R\vert^{1/n},
\end{equation}
as desired.
\end{proof}
\noindent We get a similar result for higher-order Orlicz-Sobolev embedding as well:
\begin{theorem}\label{main-higher}
Let $1\leq m<n.$ Let $A$ be any Young function satisfying \eqref{finite} and $A{\frac{n}{m}}$ be the Young function defined by \eqref{A_nm}.
Let $\Omega$ be any open set in $\mathbb{R}^n$ and $W^{m,A}(\Omega)\hookrightarrow L^{A_{\frac{n}{m}}}(\Omega).$ If $I_{A}<\frac{n}{m},$ then $\Omega$ satisfies the measure density condition.
\end{theorem}
\begin{proof}
The proof is very similar to that of the previous theorem, we will present only the main steps. For a fixed $x\in\Omega,$ let $R,\tilde{R}$ be the same as before and let $u(y):=\eta(y-x)$ be a function of $y\in\Omega$ where $\eta$ is a cut-off function satisfying:\\
1. $\eta: \mathbb{R}^n\rightarrow [0,1],$\\
2. $\spt \eta \subset B_R,$\\
3. $\eta\vert_{B_{\tilde{R}}}=1,$ and\\
4. $\vert\nabla \eta\vert\leq \tilde{c}/(R-\tilde{R})^{\vert \alpha\vert}$ for some constant $\tilde{c}$ and for all multi-indices $\alpha$ with $\vert\alpha\vert\leq m.$\\
The existence of such function is guaranteed by \cite[Lemma 10]{HKT08b}. Using this function, the embedding yields some positive constant $c_4$ such that
\begin{equation}\label{R-higher}
(R-\tilde{R})^m\leq c_4\frac{A_{\frac{n}{m}}^{-1}\left(\frac{2}{\vert B_{R}\vert}\right)}{A^{-1}\left(\frac{1}{\vert B_R\vert}\right)}.
\end{equation}  
The proof finishes with the same procedure as before after applying the following lemma.
\end{proof}
\begin{lemma}
Let $A$ and $A_{\frac{n}{m}}$ be the same as Theorem \ref{main-higher}. 
Then constants $r_0$ and $c_0$ exist such that
\begin{equation}
\frac{A_{\frac{n}{m}}^{-1}(2r)}{A^{-1}(r)}\leq c_0\left(\frac{1}{r}\right)^{m/n},
\end{equation}
for all $r\geq r_0.$
\end{lemma}
\begin{proof}
For $s\geq 0,$ define
\begin{equation*}
\Phi_{\frac{n}{m}}(s)=\int_0^s\frac{\tilde{A}(t)}{t^{1+n/(n-m)}}\,dt\quad\text{and}\quad E_{\frac{n}{m}}(s)=\int_0^s\frac{C_{\frac{n}{m}}(t)}{t}\,dt,
\end{equation*}
where $C_{\frac{n}{m}}(s)=s^{\frac{n}{n-m}}\left(\Phi_{\frac{n}{m}}^{-1}\left(s^{\frac{n}{n-m}}\right)\right)^{\frac{n}{n-m}}.$ Then, by \cite[Lemma 2]{Cia00} we know that the condition \eqref{finite} is equivalent to $\Phi_{\frac{n}{m}}(s)<\infty$ and two constants $c_1=c_1(\frac{n}{m})$ and $c_2=c_2(\frac{n}{m})$ exist such that
\begin{equation*}
E_{\frac{n}{m}}(c_1s)\leq A_{\frac{n}{m}}(s)\leq E_{\frac{n}{m}}(c_2s), \quad\text{for}\,\, s\geq 0.
\end{equation*}
Therefore, for $r\geq 0,$
\begin{equation}\label{onehand}
c_1A_{\frac{n}{m}}^{-1}(r)\leq E_{\frac{n}{m}}^{-1}(r)\leq c_2A_{\frac{n}{m}}^{-1}(r).
\end{equation}
On the other hand, the same technique as in Lemma \ref{mainlemma} gives us the existence of two constants $c_0$ and $r_0$ such that
\begin{equation}\label{otherhand}
\frac{E_{\frac{n}{m}}^{-1}(2r)}{A^{-1}(r)}\leq c_0\left(\frac{1}{r}\right)^{m/n},\quad\text{for}\,\, r\geq r_0.
\end{equation}
Combining the last two inequalities we obtain the required result.
\end{proof}
\def\bibname{References}
\bibliography{measure_density_orlicz}
\bibliographystyle{plain}
\end{document}